\documentclass[aps,11pt,twoside, nofootinbib, superscriptaddress]{revtex4}
\usepackage{amsthm}
\usepackage{amsmath,latexsym,amssymb,verbatim,enumerate,graphicx}
\usepackage{color}
\usepackage[utf8]{inputenc}
\usepackage[T1]{fontenc}
\usepackage[polish]{babel}

\newtheorem{definition}{Definition} 

\newtheorem{lemma}[definition]{Lemma}

\newtheorem{thm}[definition]{Theorem}

\newtheorem*{rep@theorem}{\rep@title}
\newcommand{\newreptheorem}[2]{%
\newenvironment{rep#1}[1]{%
 \def\rep@title{#2 \ref{##1} (restatement)}%
 \begin{rep@theorem}}%
 {\end{rep@theorem}}}
\makeatother

\newreptheorem{thm}{Theorem}
\newreptheorem{lem}{Lemma}

\begin{document}
\title{Exponential rate of convergence for some Markov operators}
\author{Hanna Wojew\'odka}
\affiliation{Institute of Mathematics, University of Gda\'nsk, Wita Stwosza
  57, 80-952 Gda\'nsk, Poland}
\email{hwojewod@mat.ug.edu.pl}

\begin{abstract} {The exponential rate of convergence for some Markov operators is established. The operators correspond to continuous iterated function systems which are a very useful tool in some cell cycle models.}
\end{abstract}


\maketitle

\section{Introduction}
We are concerned with Markov operators corresponding to continuous iterated function systems. The main purpose of~the paper is to prove spectral gap assuring exponential rate of~convergence. The operators under consideration were used in \citet{lasotam}, where the authors studied some cell cycle model. See also \citet{tysonhannsgen} or \citet{murrayhunt} to get more details on the subject. Lasota and Mackey proved only stability, while we managed to evaluate rate of~convergence, bringing some information important from biological point of~view. In our paper we base on coupling methods introduced in \citet{hairer}. In the same spirit, exponential rate of~convergence was proved in \citet{sleczka} for classical iterated function systems (see also \citet{hairerm} or \citet{kapica}). It is worth mentioning here that our result will allow us to show the Central Limit Theorem (CLT) and the Law of~Iterated Logarithm (LIL). To do this, we will adapt general results recently proved in \citet{bms} or in \citet{komorowskiwalczuk}. The proof of~CLT and LIL will be provided in a future paper.

The organization of~the paper goes as follows. Section $2$ introduces basic notation and definitions that are needed throughout the paper. Most of~them are adapted from \citet{billingsley}, \citet{tweedie}, \citet{lasotay} and \citet{szarek}. Biological background is shortly presented in~Section $3$. Sections $4$ and $5$ provide the mathematical derivation of~the model and the main theorem (Theorem \ref{main}), which establishes the exponential rate of~convergence in the model. Sections $6$-$8$ are devoted to the construction of~coupling measure for iterated function systems. Thanks to the results presented in Section $9$ we are finally able to present the proof of~the main theorem in~Section $10$.

\section{Notation and basic definiotions}

Let $(X,\varrho)$ be a Polish space. We denote by $B_X$ the family of all Borel subsets of $X$. 
Let $C(X)$ be the space of all bounded and continuous functions $f:X\to R$ with the supremum norm. 

We denote by $M(X)$ the family of all Borel measures on $X$ and by $M_{fin}(X)$ and $M_1(X)$ its subfamilies such that $\mu(X)<\infty$ and~$\mu(X)=1$, respectively. 
Elements of $M_{fin}(X)$ which satisfy $\mu(X)\leq 1$ are called sub-probability measures. To simplify notation, we write
\[\langle f,\mu\rangle =\int_X f(x)\mu(dx)\quad\text{for}\: f\in C(X),\: \mu\in M(X).\]
An operator $P:M_{fin}(X)\to M_{fin}(X)$ is called a Markov operator if
\begin{enumerate}
\item $P(\lambda_1\mu_1+\lambda_2\mu_2)=\lambda_1 P\mu_1+\lambda_2 P\mu_2\quad\text{ for }\:\lambda_1,\lambda_2\geq 0,\; \mu_1, \mu_2\in M_{fin}(X)$;
\item $P\mu(X)=\mu(X)\quad\text{ for \:$\mu\in M_{fin}(X)$}$.
\end{enumerate}
If, additionally, there exists a linear operator $U:C(X)\to C(X)$ such that
\[\langle Uf,\mu\rangle =\langle f,P\mu\rangle \quad\text{for}\: f\in C(X),\:\mu\in M_{fin}(X),\]
an operator $P$ is called a Feller operator. Every Markov operator $P$ may be extended to the space of signed measures on $X$ denoted by $M_{sig}(X)=\{\mu_1-\mu_2:\; \mu_1,\mu_2\in M_{fin}(X)\}$. For $\mu\in M_{sig}(X)$ we denote by $\|\mu\|$ the total variation norm of $\mu$, i.e.
\[\|\mu\|=\mu^+(X)+\mu^-(X),\]
where $\mu^+$ and $\mu^-$ come from the Hahn-Jordan decomposition of $\mu$ (see \citet{halmos}). 
For fixed $\bar{x}\in X$ we also consider the space $M_1^1(X)$ of all probability measures with the first moment finite, i.e. $M_1^1(X)=\{\mu\in M_1(X):\:\int_X\varrho(x,\bar{x})\mu(dx)<\infty\}$. The family is idependent of the choice of $\bar{x}\in X$. 
We call $\mu_*\in M_{fin}(X)$ an invariant measure of $P$ if $P\mu_*=\mu_*$. 
For $\mu\in M_{fin}(X)$ we define the support of $\mu$ by
\[\text{supp }\mu=\{x\in X:\; \mu(B(x,r))>0\quad \text{for }\:r>0\},\]
where $B(x,r)$ is the open ball in $X$ with center at $x\in X$ and radius $r>0$.

In $M_{sig}(X)$ we introduce the Fourtet-Mourier norm
\[\|\mu\|_{\mathcal{L}}=\sup_{f\in\mathcal{L}}|\langle f,\mu\rangle |,\]
where 
\begin{align}\label{def:mathcal_L}
\mathcal{L}=\{f\in C(X):\;|f(x)-f(y)|\leq\varrho(x,y),\;|f(x)|\leq 1\;\text{ for }\:x,y\in X\}.
\end{align}
The space $M_1(X)$ with the metric $\|\mu_1-\mu_2\|_{\mathcal{L}}$ is complete (see \citet{fortet} or \citet{rachev}). 

\section{Shortly about the model of cell division cycle}

Let $(\Omega,\mathcal{F},\text{Prob})$ be a probability space. 
Suppose that each cell in a considered population consists of $d$ different substances, whose masses are described by the vector 
$y(t)=(y^1(t),\ldots,y^d(t))$, where $t\in[0,T]$ denotes an age of a cell. 
We assume that the evolution of the vector $y(t)$ is given by the formula $y(t)=\Pi(x,t)$, where $\Pi(x,0)=x$. Here $\Pi:X\times[0,T)\to X$ is a given function. A~simple example fulfilling these criteria is given by assuming that $y(t)$ satisfies a system of ordinary differential equations
\begin{equation}\label{diffeq}
\frac{dy}{dt}=g(t,y)
\end{equation}
with the initial condition $y(0)=x$ and the solution of $(\ref{diffeq})$ is given by $y(t)=\Pi(x,t)$.

If $x_n$ denotes the initial value $x=y(0)$ of substances in the $n$-th generation and $t_n$ denotes the mitotic time in the $n$-th generation, the distribution is given by 
\begin{equation}\label{def:p}
\text{Prob}(t_n\in I|x_n=x)=\int_Ip(x,s)ds\quad \text{ for }I\in[0,T],\:n\in N\text{.}
\end{equation}
The vector $y(t_n)=\Pi(x_n,t_n)$ with $y(0)=\Pi(x,0)=x$ describes an amount of intercellular substance just before cell division in the $n$-th generation. We assume that each daughter cell contains exactly half of the components of its stem cell. Hence
\begin{equation}\label{def:Pi}
x_{n+1}=\frac{1}{2}\Pi(x_n,t_n)\quad \text{ for }\:n=0,1,2,\ldots\:\text{.}
\end{equation}
The bahaviour of $(\ref{def:p})$ and $(\ref{def:Pi})$ may be also described by the sequence $(\mu_n)_{n\geq 1}$ of distributions
\[\mu_n(A)=\text{Prob}(x_n\in A) \quad\text{ for }\: n=0,1,2,\ldots\:\text{ and }\:A\in B_{X}\text{.}\]
See \citet{lasotam} for more details.

\section{Assumptions}\label{sec:assumptions}

We assume that $(X,\varrho)$ is a Polish space. 
Fix $T<\infty$. 
We consider a family $\{t_n:\:n=0,1,\ldots\}$ of indepenent random variables taking values in $[0,T]$. The family is defined on the probability space $(\Omega,\mathcal{F},\text{Prob})$.
Note that $\text{Prob}(t_n<T|x_n=x)=1$. 
Let $S:X\times[0,T)\to X$ be a~continuous function and
\[x_{n+1}=S(x_n,t_n),\quad n=0,1,2,\ldots\:\text{.}\]
We assume that $p:X\times[0,T)\to[0,\infty)$ is a~lower semi-continuous, non-negative function such that, for every $x\in X$, $p(x,0)=0$ and $p(x,t)>0$ for $t>0$. In addition, $p$ is normalized, i.e. $\int_0^Tp(x,u)du=1$ for $x\in X$. 
Let us further assume that for each $A\in B_X$
\[\text{Prob}(x_{n+1}\in A):=\mu_{n+1}(A), \quad\text{and}\quad P\mu_n=\mu_{n+1}\text{,}\]
where
\begin{equation}\label{def:Pmu}
P\mu(A)=\int_X\left(\int_0^T 1_A(S(x,t))p(x,t)dt\right)\mu(dx)\text{.}
\end{equation}

The following assumptions will be needed throughout the paper:
\begin{itemize}
\item[(I)]  $\varrho(S(x,t),S(y,t))\leq \lambda(t)\varrho(x,y)$ for $x,y\in X$, where $\lambda:[0,T)\to[0,\infty)$ is a~Borel measurable function;
\item[(II)] $a:=\sup_{x\in X}\int_0^T\lambda(t)p(x,t)dt<1$; 
\item[(III)] $\sup_{t\in[0,T)}\varrho\left(S(\bar{x},t),\bar{x}\right)<\infty$ for some $\bar{x}\in X$;
\item[(IV)] there exists $\sigma$ such that $p:X\times[0,T)\to[\sigma,\infty)$ is a~continuous function and $\bar{c}>0$ such that $\int_0^T|p(x,t)-p(y,t)|dt\leq \bar{c}\varrho(x,y)$ for $x,y\in X$;
\item[(V)] function $p$ is bounded and we assume that $\delta=\inf\{p(x,t): x\in X, t\in(0,T)\}>0$, $M=\sup\{p(x,t):x\in X, t\in(0,T)\}$.
\end{itemize}

\section{Main theorem}

Let $P$ be the Markov operator in the cell division model defined above. Lasota and Mackey proved asymptotic stability of $P$, i.e. the existence of an invariant measure $\mu_*\in M_1(X)$ and weak convergence of $(P^n{\mu})$ to $\mu_*$ for $\mu\in M_1(X)$. The theorem says.

\begin{thm}\label{tLM}
Let $S:X\times[0,T]\to X$ and $p:X\times[0,T]\to[0,\infty)$ satisfy the following conditions
\begin{enumerate}
\item $\varrho(S(x,t),S(y,t))\leq\lambda_0(x,t)\varrho(x,y)$ for $x,y\in X$, $t\in[0,T]$ and $\lambda_0$ and $S$ related to $p$ by the conditions
$\int_0^T\lambda_0(x,t)p(x,t)dt\leq r_0$ and $\int_0^T |S(0,t)|p(x,t)dt\leq r_1$ for $x\in X$;
\item $\int_0^T|p(x,t)-p(y,t)|dt\leq r_2\varrho(x,y)$ for $x,y\in X$;
\item for every $x\in X$ there exists a minimal division time $\tau_x\in[0,T]$ such that $p(x,t)=0$ for $0\leq t\leq\tau_x$ and $p(x,t)>0$ for $\tau_x<t\leq T$.
\end{enumerate}
We assume moreover that $r_0<1$ and $r_1,r_2<\infty$. Then, the system $(\ref{def:p})$ and $(\ref{def:Pi})$ is asymptotically stable.
\end{thm}

Obviously, conditions (i) and (ii) of Theorem $\ref{tLM}$ are satisfied by assumptions (I)-(IV) of the model in consideration. Note that condition (iii) is also fulfilled with $\tau_x=0$, as for every $x\in X$ we have $p(x,0)=0$ and $p(x,t)>0$ for every $t>0$ and $x\in X$. That is why we can assume the existance of an invariant measure in the model.

Our aim is to show that rate of convergence is exponential.
\begin{thm}\label{main}
Let $\mu\in M_1^1$. Under assumptions (I)-(V) there exist $C=C(\mu)>0$ and $q\in[0,1)$ such that
\[\|P^n\mu-\mu_*\|_{\mathcal{L}}\leq Cq^n\quad \text{ for }n\in N.\]
\end{thm}

\section{Measures on the pathspace and coupling}
We consider a family of measures $\{Q_x:x\in X\}$ on $X$. We assume measurability of the mappings $x\mapsto Q_x(A)$ for each $A\in B_X$. 
Fix $n,m\in N$. Now, suppose that $\{Q_x:x\in X\}$ is a family of measures on $X^n$ and $\{R_x:x\in X\}$ is a family of measures on $X^m$. We can define a family of measures $\{(RQ)_x:x\in X\}$ on $X^n\times X^m$
\begin{equation}\label{constr:RQ}
(RQ)_x(A\times B)=\int_A R_{z_n}(B)Q_x(dz),
\end{equation}
where $z=(z_1,\ldots,z_n)$ and $A\in B_{X^n}$, $B\in B_{X^m}$.

We consider a family of sub-probability measures $\{P_x:x\in X\}$ on $X$. 
We assume that the mapping $x\mapsto P_x(A)$ is measurable for each $A\in B_X$. Furthermore, if each $P_x$ is a probability measure, $\{P_x:x\in X\}$ is a transition probability function. Thus $P_x(A)$ is the probability of transition from~$x$ to~$A$. 
We want to define a family of measures on $X^{\infty}$. Set $x\in X$. One-dimensional distributions $\{P_x^n:n\in N\}$ are defined by induction on $n$
\begin{align}\label{constr:n+1}
\begin{aligned}
P_x^0(A)&=\delta_x(A), \;\ldots,\; P_x^{n+1}(A)&=\int_X P_z(A)P_x^n(dz),
\end{aligned}
\end{align}
where $A\in B_X$.
Following $(\ref{constr:RQ})$, we easily obtain two and higher-dimentional distributions. Finally, we get the family $\{P_x^{\infty}:x\in X\}$ of sub-probability measures on $X^{\infty}$. This construction was motivated by \citet{hairer}. The existance of measures $P_x^{\infty}$ is established by the Kolmogorov theorem. More precisely, there exists some probability space, on which we can define a~stochastic proces $\xi$ with distribution $\phi_{\xi}$ such that
\[\phi_{\xi}(A)=\text{Prob}(\xi^{-1}(A)):=P_x^{\infty}(A)\quad\text{ for $A\in B_{X^{\infty}}$.}\] 
Therefore, $P_x^{\infty}$ is the distribution of the Markov chain $\xi$ on $X^{\infty}$ with transition probability function $\{P_x: x\in X\}$ and $\phi_{\xi_0}=\delta_x$ for $x\in X$. If an initial distribution is given by any $\mu\in M_{fin}(X)$, not necessarily by $\delta_x$, we define
\[P_{\mu}^{\infty}(A)=\int_X P_x^{\infty}(A)\mu(dx)\quad\text{ for $A\in B_{X^{\infty}}$}.\]

\begin{definition}
Let a transition probability function $\{P_x:x\in X\}$ be given. A~family of probability measures $\{C_{x,y}:x,y\in X\}$ on $X\times X$ such that
\begin{itemize}
\item $C_{x,y}(A\times X)=P_x(A)$ \;for $A\in B_X$,
\item $C_{x,y}(X\times B)=P_y(B)$ \;for $B\in B_X$,
\end{itemize}
where $x,y\in X$, is called coupling.
\end{definition}

\section{Iterated function systems}\label{sec:ifs}

We consider a continuous mapping $S:X\times[0,T)\to X$ and a lower semi-continuous, non-negative normalized function $p:X\times[0,T)\to[0,\infty)$. 
For each $A\in B_X$ we build a transition operator $P_x(A)=\Pi(x,A)$. 
Since $P\mu$ is given by $(\ref{def:Pmu})$ and $(P\mu)(A)=\int_X P_x(A)\mu(dx)$, we define $P_x$ to be
\[P_x(A)=\int_0^T 1_A(S(x,t))p(x,t)dt=\int_0^T\delta_{S(x,t)}(A)p(x,t)dt.\]
Once again, we apply $(\ref{constr:RQ})$ and $(\ref{constr:n+1})$ to construct measures on products. As previously, $P_{\mu}^{\infty}$ exists for $\mu\in M_{fin}(X)$. Obviously, $P^n{\mu}$ is the $n$-th marginal of $P_{\mu}^{\infty}$.

Fix $\bar{x}\in X$. We define $V:X\to[0,\infty)$ to be
\[V(x)=\varrho(x,\bar{x}).\]
Let us evalute an integral $\langle V,P\mu\rangle =\int_X\varrho(x,\bar{x})P\mu(dx)=\int_X U\varrho(x,\bar{x})\mu(dx)$, 
where $U$ is a dual operator to $P$. 
Since $P$ is a Feller operator given by $(\ref{def:Pmu})$, we can define $U:C(X)\to C(X)$ by
\[Uf(x)=\int_0^Tf(S(x,t))p(x,t)dt.\]
Hence, from initial assumptions (I) and (II), we obtain
\begin{align*}
\begin{aligned}
\langle V,P\mu\rangle &=\int_X\left(\int_0^T\varrho\left(S(x,t),\bar{x}\right)p(x,t)dt\right)\mu(dx)\\
&\leq\int_X\left(\int_0^T\left(\varrho(S(x,t),S(\bar{x},t))+\varrho(S(\bar{x},t),\bar{x})\right)p({x},t)dt\right)\mu(dx)\\
&\leq\int_X\left(\int_0^T\lambda(t)\varrho(x,\bar{x})p({x},t)dt+\int_0^T\varrho(S(\bar{x},t)\bar{x})p({x},t)dt\right)\mu(dx)\\
&\leq a\int_X\varrho(x,\bar{x})\mu(dx)+\int_X \tilde{c}\mu(dx)\\
&=a\langle V,\mu\rangle +c,
\end{aligned}
\end{align*}
where $c=\int_X \tilde{c}\mu(dx)$ and $\tilde{c}=\sup_{t\in[0,T)}\varrho(S(\bar{x},t),\bar{x})$ exists from assumption (III). 
Fix probability measures $\mu,\nu\in M_1^1(X)$ and Borel sets $A,B\in B_X$. We consider $b\in M_1(X^2)$ such that
\[b(A\times X)=\mu(A)\text{,}\qquad b(X\times B)=\nu(B)\]
and $\bar{b}\in M_1(X^2)$ such that
\[\bar{b}(A\times X)=P\mu(A)\text{,}\qquad \bar{b}(X\times B)=P\nu(B)\text{.}\]
Furthermore, we define $\bar{V}:X^2\to[0,\infty)$
\[\bar{V}(x,y)=V(x)+V(y)\quad\text{for $x,y\in X$.}\]
Note that
\begin{align}\label{prop:barV}
\langle \bar{V},\bar{b}\rangle \:\leq \:a\langle \bar{V},b\rangle +2c.
\end{align}
For measures $b\in M_{fin}^1(X^2)$ finite on $X^2$ and with the first moment finite we define the~linear functional
\[\phi(b)=\int_{X^2}\varrho(x,y)b(dx,dy).\]
Following the above definitions, we easily obtain 
\begin{align}\label{prop:phi}
\phi(b)\:\leq\:\langle \bar{V},b\rangle .
\end{align}

\section{Coupling for itereted function systems}

On $X^{\infty}$ we define the transition sub-probability function
\begin{equation}\label{def:Qxy}
Q_{x,y}(A\times B)=\int_0^T \min\{p(x,t),p(y,t)\}\delta_{(S(x,t),S(y,t))}(A\times B)dt\quad\text{ for $\:A,B\in B_{X}$}.
\end{equation}
It is easy to check that
\begin{align*}
\begin{aligned}
Q_{x,y}(A\times X)&\leq\int_0^Tp(x,t)\delta_{S(x,t)}(A)dt
&=\int_0^T1_A(S(x,t))p(x,t)dt
&=P_x(A)
\end{aligned}
\end{align*}
and analogously 
\[Q_{x,y}(X\times B)\leq P_y(B).\]
Let $Q_b$ denote the measure
\begin{equation}\label{def:Qb}
Q_b(A\times B)=\int_{X^2}Q_{x,y}(A\times B)b(dx,dy)\quad\text{ for }\:A,B\in B_{X}\text{.}
\end{equation}
Note that for every $A,B\in B_{X}$ we obtain
\begin{align*}
\begin{aligned}
Q_b^{n+1}(A\times B)&=\int_{X^2}Q_{x,y}^{n+1}(A\times B)b(dx,dy)\\
&=\int_{X^2}\int_{X^2}Q_{z_1,z_2}(A\times B)Q_{x,y}^n(dz_1,dz_2)b(dx,dy)\\
&=\int_{X^2}Q_{z_1,z_2}(A\times B)\int_{X^2}Q_{x,y}^n(dz_1,dz_2)b(dx,dy)\\
&=\int_{X^2}Q_{z_1,z_2}(A\times B)Q_b^n(dx,dy)=Q_{Q_b^n}(A\times B).
\end{aligned}
\end{align*}
Again, we are able to construct measures on products, as well as we are able to construct $Q_b^{\infty}$ on~$X^{\infty}$. 
Now, we check that
\begin{equation}\label{prop:phiQb}
\phi(Q_b)\leq a\phi(b).
\end{equation}
Let us observe that
\begin{align*}
\begin{aligned}
\phi(Q_b)
&=\int_{X^2}\int_{X^2}\varrho(x,y)Q_{u,v}(dx,dy)b(du,dv)\\
&=\int_{X^2}\int_0^T\left(\int_{X^2}\varrho(x,y)\min\{p(u,t),p(v,t)\}\delta_{(S(u,t),S(v,t))}(dx,dy)\right)dt\:b(du,dv)\\
&\leq\int_{X^2}\int_0^T\left(\varrho(S(u,t),S(v,t))p(u,t)\right)dt\:b(du,dv)\\
&\leq\int_{X^2}\int_0^T\lambda(t)\varrho(u,v)p(u,t)dt\:b(du,dv)\\
&\leq a\int_{X^2}\varrho(u,v)b(du,dv)\\
&=a\phi(b).
\end{aligned}
\end{align*}

We can find such a measure $R_{x,y}$ that the sum of $Q_{x,y}$ and $R_{x,y}$ gives a new coupling measure $C_{x,y}$, i.e. $C_{x,y}(A\times X)=P_x(A)$ and $C_{x,y}(X\times B)=P_y(B)$ for $A,B\in B_X$.

\begin{lemma}\label{lemma1}
There exists the family $\{R_{x,y}:x,y\in X\}$ of measures on~$X^2$ such that we can define
\[C_{x,y}=Q_{x,y}+R_{x,y}\quad\text{ for $x,y\in X$}\]
and, moreover, 
\begin{itemize}
\item[(i)] the mapping $(x,y)\mapsto R_{x,y}(A\times B)$ is measurable for every $A,B\in B_{X}$;
\item[(ii)] measures $R_{x,y}$ are non-negative for $x,y\in X$;
\item[(iii)] measures $C_{x,y}$ are probabilistic for every $x,y\in X$ and so $\{C_{x,y}:x,y\in X\}$ is the~transition probability function on $X^2$;
\item[(iv)] for every $A,B\in B_X$ and $x,y\in X$ we get $C_{x,y}(A\times X)=P_x(A)$ and $C_{x,y}(X\times B)=P_y(B)$.
\end{itemize}
\end{lemma}

\begin{proof}[Proof]
Fix $A,B\in B_{X}$. Let
\[R_{x,y}(A\times B)=\left\{\begin{array}{l l}
(1-Q_{x,y}(X^2))^{-1}(P_x(A)-Q_{x,y}(A\times X))(P_y(B)-Q_{x,y}(X\times B)),& Q_{x,y}(X^2)<1\\
0,& Q_{x,y}(X^2)=1.
\end{array}\right.\]
Obviously, the formula may be extended to the measure. 
The mapping has all desirable properties (i)-(iv).
\end{proof}
Lemma $\ref{lemma1}$ shows that we can construct the coupling $\{C_{x,y}:x,y\in X\}$ for $\{P_x:x\in X\}$ such that $Q_{x,y}\leq C_{x,y}$, whereas measures $R_{x,y}$ are non-negative. 
By $(\ref{constr:RQ})$ and $(\ref{constr:n+1})$ we obtain the family of probability measures $\{C_{x,y}^{\infty}:x,y\in X\}$ on $(X^2)^{\infty}$ with marginals $P_x^{\infty}$ and $P_y^{\infty}$. This construction appears in \citet{hairer}.

Fix $(x_0,y_0)\in X^2$. The transition probability function $\{C_{x,y}:x,y\in X\}$ defines the Markov chain $\Phi$ on $X^2$ with starting point $(x_0,y_0)$, while the transition probability function $\{\hat{C}_{x,y,\theta}:x,y\in X, \theta\in\{0,1\}\}$ defines the Markov chain $\hat{\Phi}$ on the augmented space $X^2\times\{0,1\}$ with initial distribution $\hat{C}_{x_0,y_0}^0=\delta_{(x_0,y_0,1)}$. If $\hat{\Phi}_n=(x,y,i)$,
where $x,y\in X$, $i\in\{0,1\}$, then
\[\text{Prob}(\hat{\Phi}_{n+1}\in A\times B\times\{1\}\:|\:\hat{\Phi}_n=(x,y,i),i\in\{0,1\})=Q_{x,y}(A\times B),\]
\[\text{Prob}(\hat{\Phi}_{n+1}\in A\times B\times\{0\}\:|\:\hat{\Phi}_n=(x,y,i),i\in\{0,1\})=R_{x,y}(A\times B),\]
where $A,B\in B_X$. 
Once again, we refer to $(\ref{constr:RQ})$ and $(\ref{constr:n+1})$ to obtain the measure $\hat{C}^{\infty}_{x_0,y_0}$ on $(X^2\times\{0,1\})^{\infty}$ which is associated with the Markov chain $\hat{\Phi}$.

From now on, we assume that processes $\Phi$ and $\hat{\Phi}$ taking values in $X^2$ and $X^2\times\{0,1\}$, respectively, are defined on $(\Omega, F, \mathbf{P})$. The expected value of the measures $C_{x_0,y_0}^{\infty}$ or $\hat{C}_{x_0,y_0}^{\infty}$ is denoted by $E_{x_0,y_0}$.

\section{Auxiliary theorems}

Fix $\varepsilon\in(0,1-a)$. Set
\[K_{\varepsilon}=\{(x,y)\in X^2: \: \bar{V}(x,y)<\varepsilon^{-1}2c\},\]
where $c$ is defined in Section \ref{sec:ifs}. Let $d:(X^2)^{\infty}\to N$ denote the time of the first visit in $K_{\varepsilon}$, i.e.
\[d((x_n,y_n)_{n\in N})=\inf\{n\geq 1:\: (x_n,y_n)\in K_{\varepsilon}\}.\]

\begin{thm}\label{theorem1}
For every $\gamma\in(0,1)$ there exist positive constants $C_1,C_2$ such that
\[E_{x_0,y_0}\left((a+\varepsilon)^{-\gamma d}\right)\leq C_1\bar{V}(x_0,y_0)+C_2.\]
\end{thm} 

\begin{proof}[Proof]
Fix $(x_0,y_0)\in X^2$. Let $\Phi=(X_n,Y_n)_{n\in N}$ be the Markov chain with starting point $(x_0,y_0)$ and transition probability function $\{C_{x,y}:x,y\in X\}$. Let $F_n\subset F$, $n\in N$ be the natural filtration in $\Omega$ associated with $\Phi$. We define
\[A_n=\{\omega\in\Omega: \; \Phi_i=(X_i(\omega), Y_i(\omega))\notin K_{\varepsilon}\; \text{ for }\:i=1,\ldots,n\}, \quad n\in N.\]
Obviously $A_{n+1}\subset A_n$ and $A_n\in F_n$ for $n\in N$. 
The following inequalities are $\mathbf{P}$-a.s. satisfied in $\Omega$
\[1_{A_n}E_{x_0,y_0}\left(\bar{V}(X_{n+1},Y_{n+1})|F_n\right)\leq 1_{A_n}(a\bar{V}(X_n,Y_n)+2c)\leq 1_{A_n}(a+\varepsilon)\bar{V}(X_n,Y_n).\]
The first inequality is a consequence of $(\ref{prop:barV})$, the second follows directly from the definitions of $A_n$ and $K_{\varepsilon}$. 
Accordingly, we obtain
\begin{align*}
\begin{aligned}
\int_{A_n}\bar{V}(X_n,Y_n)d\mathbf{P}&\leq\int_{A_{n-1}}\bar{V}(X_n,Y_n)d\mathbf{P}
=\int_{A_{n-1}}E\left(\bar{V}(X_n,Y_n)|F_{n-1}\right)d\mathbf{P}\\
&\leq\int_{A_{n-1}}\left(a\bar{V}(X_{n-1},Y_{n-1})+2c\right)d\mathbf{P}
\leq(a+\varepsilon)\int_{A_{n-1}}\bar{V}(X_{n-1},Y_{n-1})d\mathbf{P}.
\end{aligned}
\end{align*}
On applying this estimates finitely many times, we obtain
\[\int_{A_n}\bar{V}(X_n,Y_n)d\mathbf{P}\leq(a+\varepsilon)^{n-1}\int_{A_1}\bar{V}(X_1,Y_1)d\mathbf{P}\leq(a+\varepsilon)^{n-1}\left(a\bar{V}(X_0,Y_0)+2c\right).\]
Note that
\begin{align*}
\begin{aligned}
\mathbf{P}(A_n)&\leq\int_{A_n}\varepsilon(2c)^{-1}\bar{V}(X_n,Y_n)d\mathbf{P}
&\leq\varepsilon\left(2c(a+\varepsilon)\right)^{-1}(a+\varepsilon)^n\left(a\bar{V}(X_0,Y_0)+2c\right).
\end{aligned}
\end{align*}
Set $\hat{c}:=\varepsilon\left(2c(a+\varepsilon)\right)^{-1}\left(a\bar{V}(X_0,Y_0)+2c\right)$. Then, $\mathbf{P}(A_n)\leq(a+\varepsilon)^n\hat{c}$. 
Fix $\gamma\in(0,1)$. Since $d$ takes natural values $n\in N$, we obtain
\begin{align*}
\begin{aligned}
\sum_{n=1}^{\infty}(a+\varepsilon)^{-\gamma n}\mathbf{P}(A_n)
&\leq\sum_{n=1}^{\infty}(a+\varepsilon)^{-\gamma n}(a+\varepsilon)^n\hat{c}
&=\sum_{n=1}^{\infty}(a+\varepsilon)^{(1-\gamma)n}\hat{c},\\
\end{aligned}
\end{align*}
which implies convergence of the series. The proof is complete by the definition of $\hat{c}$ and with properly choosen $C_1$, $C_2$.
\end{proof}
For every positive $r>0$ we determine the set
\[C_r=\{(x,y)\in X^2:\; \varrho(x,y)<r\}.\]

\begin{lemma}\label{lemma2}
Fix $\tilde{a}\in(a,1)$. Let $C_r$ be the set defined above and suppose that $\text{supp }b\subset C_r$. There exists $\bar{\gamma}>0$ such that
\[Q_b(C_{\tilde{a}r})\geq\bar{\gamma}\|b\|\]
for $a$, $\delta$ and $M$ defined in Section $\ref{sec:assumptions}$.
\end{lemma}

\begin{proof}[Proof]
Directly from $(\ref{def:Qb})$ and $(\ref{def:Qxy})$ we obtain
\begin{align*}
\begin{aligned}
Q_b(C_{\tilde{a}r})&=\int_{X^2}\int_0^T\min\{p(x,t),p(y,t)\}\delta_{(S(x,t),S(y,t))}(C_{\tilde{a}r})dt \: b(dx,dy)\\
&=\int_{X^2}\left(\int_0^Tmin\{p(x,t),p(y,t)\}1_{C_{\tilde{a}r}}(S(x,t),S(y,t))dt\right)b(dx,dy).\\
\end{aligned}
\end{align*}
Note that $1_{C_{\tilde{a}r}}(S(x,t),S(y,t))=1$ if and only if $t\in\mathcal{T}$, where \[\mathcal{T}:=\{t\in(0,T):\varrho(S(x,t),S(y,t))<\tilde{a}r\}.\] Set $\mathcal{T'}:=(0,T)\backslash\mathcal{T}$. Hence
\begin{align*}
Q_b(C_{\tilde{a}r})=\int_{X^2}\Big(&\int_{\mathcal{T}}\min\{p(x,t),p(y,t)\}1_{C_{\tilde{a}r}}(S(x,t),S(y,t))dt\\
&+\int_{\mathcal{T'}}\min\{p(x,t),p(y,t)\}1_{C_{\tilde{a}r}}(S(x,t),S(y,t))dt\Big)b(dx,dy).
\end{align*}
Note that
\[\int_{\mathcal{T'}}\min\{p(x,t),p(y,t)\}\varrho(S(x,t),S(y,t))dt\leq\int_{\mathcal{T'}}p(x,t)\lambda(t)\varrho(x,y)dt\leq a\varrho(x,y),\]
so for $(x,y)\in C_r$
\[\int_{\mathcal{T'}}\min\{p(x,t),p(y,t)\}\varrho(S(x,t),S(y,t))dt\leq ar.\]
However,
\[\tilde{a}r\int_{\mathcal{T'}}p(x,t)dt<\int_{\mathcal{T'}}p(x,t)\varrho(S(x,t),S(y,t))dt.\]
Therefore
\[\int_{\mathcal{T'}}p(x,t)dt<\frac{a}{\tilde{a}}<1,\]
which implies that the first integral is non-zero. Furthermore, the length of $\mathcal{T'}$ satisfies $|\mathcal{T'}|<a(\tilde{a}\delta)^{-1}$. 
We obtain
\[\int_{\mathcal{T}}p(x,t)dt\geq 1-\frac{a}{\tilde{a}}=\gamma,\]
which means that $|\mathcal{T}|\geq M^{-1}\gamma$. Finally,
\begin{align*}
\begin{aligned}
Q_b(C_{\tilde{a}r})&\geq\int_{X^2}\int_{\mathcal{T}}\min\{p(x,t),p(y,t)\}1_{C_{\tilde{a}r}}(S(x,t),S(y,t))dt\: b(dx,dy)\\
&\geq\int_{X^2}\delta|\mathcal{T}|b(dx,dy)
\geq\delta\frac{\gamma}{M}\|b\|.
\end{aligned}
\end{align*}
If we set $\bar{\gamma}:=\delta M^{-1}\gamma$, the proof is complete.
\end{proof}

\begin{thm}\label{theorem2}
For every $\varepsilon\in(0,1-a)$ there exists $n_0\in N$ such that
\[\|Q_{x,y}^{\infty}\|\geq\frac{1}{2}\bar{\gamma}^{n_0}\quad\text{for $(x,y)\in K_{\varepsilon}$,}\]
where $\bar{\gamma}>0$ is given in Lemma \ref{lemma2}.
\end{thm}

\begin{proof}[Proof]
Note that for every $(x,y)\in X^2$
\[\int_0^T\left(\min\{p(x,t),p(y,t)\}+|p(x,t)-p(y,t)|-p(x,t)\right)dt\geq 0,\]
and therefore
\[\|Q_{x,y}\|+\int_0^T|p(x,t)-p(y,t)|dt\geq 1.\]
From assumption (IV) there is $\bar{c}>0$ such that
\[\|Q_{x,y}\|\geq 1-\int_0^T|p(x,t)-p(y,t)|dt\geq 1-\bar{c}\varrho(x,y).\]
For every $b\in M_{fin}(X^2)$ we get
\begin{align*}
\begin{aligned}
\|Q_b\|&=\int_{X^2}\|Q_{x,y}\|b(dx,dy)
&\geq\int_{X^2}b(dx,dy)-\bar{c}\int_{X^2}\varrho(x,y)b(dx,dy)
&=\|b\|-\bar{c}\phi(b).
\end{aligned}
\end{align*}
Property $(\ref{prop:phiQb})$ implies that
\[\|Q_b^{n+1}\|\geq\|b\|-\bar{c}(\sum_{k=0}^na^k)\phi(b)\geq \|b\|-(1-a)^{-1}\bar{c}\phi(b),\quad n\in N.\]
If $\text{supp }b\subset C_r$, then
\begin{align*}
\begin{aligned}
\phi(b)
&\leq\int_{C_r}\varrho(x,y)b(dx,dy)
&\leq r\|b\|.
\end{aligned}
\end{align*}
Let $r=(2\bar{c})^{-1}(1-a)$. We obtain
\[\|Q_b^{\infty}\|\geq\frac{\|b\|}{2}.\]
Fix $\varepsilon\in(0,1-a)$. It is clear that $K_{\varepsilon}\subset C_{\varepsilon^{-1} 2c}$. If we define 
$n_0:=\min\{n\geq 1:\: a^n(\varepsilon)^{-1}2c<r\}$, 
then $C_{a^{n_0}\varepsilon^{-1} 2c}\subset C_r$. 
Remembering that $Q_{x,y}^{n+m}=Q^m_{Q_{x,y}^n}$ and using the Markov property, we obtain
\[Q_{x,y}^{\infty}(X^2)\geq Q_{Q_{x,y}^{n_0}}^{\infty}(X^2).\]
According to Lemma $\ref{lemma2}$, we obtain
\begin{align*}
\begin{aligned}
\|Q_{x,y}^{\infty}\|&\geq\|Q_{Q_{x,y}^{n_0}}^{\infty}\|&\geq\frac{\|Q_{x,y}^{n_0}\|}{2}
&=\frac{Q_{x,y}^{n_0}(C_r)}{2}
&\geq\frac{Q_{x,y}^{n_0}(C_{a^{n_0}\varepsilon^{-1} 2c})}{2}
&\geq\frac{\bar{\gamma}^{n_0}}{2}
\end{aligned}
\end{align*}
for $(x,y)\in K_{\varepsilon}$. This finishes the proof.
\end{proof}

\begin{definition}
Coupling time $\tau:(X^2\times\{0,1\})^{\infty}\to N$ is defined as follows
\[\tau((x_n,y_n,\theta_n)_{n\in N})=\inf\{n\geq 1:\;\theta_k=1\;\text{ for }k\geq n\}.\] 
\end{definition}

\begin{thm}
There exist $\tilde{q}\in(0,1)$ and $C_3>0$ such that
\[E_{x,y}\left(\tilde{q}^{-\tau}\right)\leq C_3(1+\bar{V}(x,y))\quad \text{for $(x,y)\in X^2$.}\]
\end{thm}

\begin{proof}[Proof]
Fix $\varepsilon\in(0,1-a)$ and $(x,y)\in X$. To simplify notation, we write $\beta=(a+\varepsilon)^{\frac{1}{2}}$. 
Let $d$ be the random moment of the first visit in $K_{\varepsilon}$. 
Suppose that
\[d_1=d,\quad d_{n+1}=d_n+d\circ T_{d_n},\]
where $n>1$ and $T_n$ are shift operators on $(X^2\times\{0,1\})^{\infty}$, i.e. 
$T_n((x_k,y_k,\theta_k)_{k\in N})=(x_{k+n},y_{k+n},\theta_{k+n})_{k\in N}$.
Theorem $\ref{theorem1}$ implies that every $d_n$ is $C_{x,y}^{\infty}$-a.s. finished. The strong Markov property shows that
\[E_{x,y}\left(\beta^d\circ T_{d_n}|F_{d_n}\right)=E_{(X_{d_n},Y_{d_n})}\left(\beta^d\right)\quad \text{for }n\in N,\]
where $F_{d_n}$ denotes the $\sigma$-algebra on $(X^2\times\{0,1\})$ generated by $d_n$ and $\Phi=(X_n,Y_n)_{n\in N}$ is the Markov chain with transition probability function $\{C_{x,y}^{\infty}:x,y\in X\}$. 
By Theorem $\ref{theorem1}$ and the definition of $K_{\varepsilon}$ we obtain
\[E_{x,y}\left(\beta^{d_{n+1}}\right)=E_{x,y}\left(\beta^{d^n}E_{(X_{d_n},Y_{d_n})}\left(\beta^d\right)\right)\leq E_{x,y}\left(\beta^{d_n}\right)(C_1\varepsilon^{-1} 2c+C_2).\]
Fix $\eta=C_1\varepsilon^{-1} 2c+C_2$. Consequently,
\begin{align}\label{condition1}
E_{x,y}\left(\beta^{d_{n+1}}\right)\leq\eta^n E_{x,y}\left(\beta^d\right)\leq\eta^n\left(C_1\bar{V}(x,y)+C_2\right).
\end{align}
We define $\hat{\tau}((x_n,y_n,\theta_n)_{n\in N})=\inf\{n\geq 1:\; (x_n,y_n)\in K_{\varepsilon}, \ \;\theta_k=1\:\text{ for }k\geq n\}$ 
and $\sigma=\inf\{n\geq 1:\;\hat{\tau}=d_n\}$. 
By Theorem $\ref{theorem2}$ there is $n_0\in N$ such that
\begin{align}\label{condition2}
\hat{C}_{x,y}^{\infty}(\sigma>n)\leq(1-\frac{\bar{\gamma}^{n_0}}{2})^n \quad\text{for }n\in N.
\end{align}
Let $d>1$. By the H\"{o}lder inequality, $(\ref{condition1})$ and $(\ref{condition2})$ we obtain
\begin{align*}
\begin{aligned}
E_{x,y}\left(\beta^{\frac{\hat{\tau}}{p}}\right)&\leq\sum_{k=1}^{\infty}E_{x,y}\left(\beta^{\frac{d_k}{p}}1_{\sigma=k}\right)
\leq\sum_{k=1}^{\infty}\left(E_{x,y}\left(\beta^{d_k}\right)\right)^{\frac{1}{p}}\left(\hat{C}_{x,y}^{\infty}(\sigma=k)\right)^{(1-\frac{1}{p})}\\
&\leq\left(C_1\bar{V}(x,y)+C_2\right)^{\frac{1}{p}}\eta^{-\frac{1}{p}}\sum_{k=1}^{\infty}\eta^{\frac{k}{p}}(1-\frac{1}{2}\bar{\gamma}^{n_0})^{(k-1)(1-\frac{1}{p})}\\
&=\left(C_1\bar{V}(x,y)+C_2\right)^{\frac{1}{p}}\eta^{-\frac{1}{p}}(1-\frac{1}{2}\bar{\gamma}^{n_0})^{-(1-\frac{1}{p})}\sum_{k=1}^{\infty}\left(\left(\frac{\eta}{1-\frac{1}{2}\bar{\gamma}^{n_0}}\right)^{\frac{1}{p}}(1-\frac{1}{2}\bar{\gamma}^{n_0})\right)^k.
\end{aligned}
\end{align*}
For $p$ sufficiently large and $\tilde{q}=\beta^{-\frac{1}{p}}$, we get
\[E_{x,y}\left(\tilde{q}^{-\hat{\tau}}\right)=E_{x,y}\left(\beta^{\frac{\hat{\tau}}{p}}\right)\leq(1+\bar{V}(x,y))C_3\]
for some $C_3$. Since $\tau\leq\hat{\tau}$, we finish the proof.
\end{proof}

\begin{thm}\label{theorem4}
There exist $q\in(0,1)$ and $C_6>0$ such that
\[\|P_x^n-P_y^n\|_{\mathcal{L}}\leq q^nC_6(1+\bar{V}(x,y))\quad\text{for }\:x,y\in X\;\text{ and }\:n\in N.\]
\end{thm}

\begin{proof}[Proof]
For $n\in N$ we define sets
\[A_{\frac{n}{2}}=\{t\in(X^2\times\{0,1\})^{\infty}:\:\tau(t)\leq\frac{n}{2}\},\]
\[B_{\frac{n}{2}}=\{t\in(X^2\times\{0,1\})^{\infty}:\:\tau(t)>\frac{n}{2}\}.\]
Note that $A_{\frac{n}{2}}\cap B_{\frac{n}{2}}=\emptyset$ and $A_{\frac{n}{2}}\cup B_{\frac{n}{2}}=(X^2\times\{0,1\})^{\infty}$, so for $n\in N$ we have
\[\hat{C}_{x,y}^{\infty}=\hat{C}_{x,y}^{\infty}|_{A_{\frac{n}{2}}}+\hat{C}_{x,y}^{\infty}|_{B_{\frac{n}{2}}}.\]
Hence,
\begin{align*}
\begin{aligned}
\|P_x^n-P_y^n\|_{\mathcal{L}}&=\sup_{f\in\mathcal{L}}|\int_{X^2}f(z)(P_x^n-P_y^n)(dz)|
&=\sup_{f\in\mathcal{L}}|\int_{X^2}(f(z_1)-f(z_2))(\Pi_{X^2}^*\Pi_n^*\hat{C}_{x,y}^{\infty})(dz_1,dz_2)|,
\end{aligned}
\end{align*}
where $\Pi_n:(X^2\times\{0,1\})^{\infty}\to X^2\times\{0,1\}$ are the projections on the $n$-th component and $\Pi_{X^2}:X^2\times\{0,1\}\to X^2$ is the projection on $X^2$. Now, recalling the definition of the set $\mathcal{L}$ (see (\ref{def:mathcal_L})), 
 we obtain
\begin{align*}
\begin{aligned}
\|P_x^n-P_y^n\|_{\mathcal{L}}&=\sup_{f\in\mathcal{L}}\Big|\int_{X^2}(f(z_1)-f(z_2))(\Pi_{X^2}^*\Pi_n^*\hat{C}_{x,y}^{\infty}|_{A_{\frac{n}{2}}})(dz_1,dz_2)\\
&\qquad+\int_{X^2}(f(z_1)-f(z_2))(\Pi_{X^2}^*\Pi_n^*\hat{C}_{x,y}^{\infty}|_{B_{\frac{n}{2}}})(dz_1,dz_2)\Big|\\
&\leq\sup_{f\in\mathcal{L}}\Big|\int_{X^2}(f(z_1)-f(z_2))(\Pi_{X^2}^*\Pi_n^*\hat{C}_{x,y}^{\infty}|_{A_{\frac{n}{2}}})(dz_1,dz_2)\Big|+2\hat{C}_{x,y}^{\infty}(B_{\frac{n}{2}})\\
&\leq\sup_{f\in\mathcal{L}}\Big|\int_{X^2}\varrho(z_1,z_2)(\Pi_{X^2}^*\Pi_n^*\hat{C}_{x,y}^{\infty}|_{A_{\frac{n}{2}}})(dz_1,dz_2)\Big|+2\hat{C}_{x,y}^{\infty}(B_{\frac{n}{2}}).
\end{aligned}
\end{align*}
Note that by iterative application of $(\ref{prop:phiQb})$ we obtain
\begin{align*}
\begin{aligned}
\int_{X^2}\varrho(z_1,z_2)(\Pi_{X^2}^*\Pi_n^*\hat{C}_{x,y}^{\infty}|_{A_{\frac{n}{2}}})(dz_1,dz_2)=\phi(\Pi_{X^2}^*\Pi_n^*(\hat{C}_{x,y}^{\infty}|_{A_{\frac{n}{2}}}))&\leq a^{\frac{n}{2}}\phi(\Pi_{X^2}^*\Pi_{\frac{n}{2}}^*(\hat{C}_{x,y}^{\infty}|_{A_{\frac{n}{2}}})).
\end{aligned}
\end{align*}
Then it follows from $(\ref{prop:barV})$ and $(\ref{prop:phi})$ that
\begin{align*}
\begin{aligned}
\phi(\Pi_{X^2}^*\Pi_{\frac{n}{2}}^*(\hat{C}_{x,y}^{\infty}|_{A_{\frac{n}{2}}}))\leq a^{\frac{n}{2}}\bar{V}(x,y)+\frac{2c}{1-a}
\end{aligned}
\end{align*}
We obtain coupling inequality
\[\|P_x^n-P_y^n\|_{\mathcal{L}}\leq a^{\frac{n}{2}}\left(a^{\frac{n}{2}}\bar{V}(x,y)+\frac{2c}{1-a}\right)+2\hat{C}_{x,y}^{\infty}(B_{\frac{n}{2}}).\]
It follows from Theorem $\ref{theorem4}$ and the Chebyshev inequality that
\begin{align*}
\begin{aligned}
\hat{C}_{x,y}^{\infty}(B_{\frac{n}{2}})&=\hat{C}_{x,y}^{\infty}(\{\tau>\frac{n}{2}\})
&=\hat{C}_{x,y}^{\infty}(\{\tilde{q}^{-\tau}\leq\tilde{q}^{-\frac{n}{2}}\})
&\leq\frac{E_{x,y}\left(\tilde{q}^{-\tau}\right)}{\tilde{q}^{-\frac{n}{2}}}&\leq\tilde{q}^{\frac{n}{2}}C_4(1+\bar{V}(x,y))
\end{aligned}
\end{align*}
for some $\tilde{q}\in(0,1)$ and $C_4>0$. 
Finally,
\[\|P_x^n-P_y^n\|_{\mathcal{L}}\leq a^{\frac{n}{2}}C_5(1+\bar{V}(x,y))+2\tilde{q}^{\frac{n}{2}}C_4(1+\bar{V}(x,y)),\]
where $C_5=\max\{a^{\frac{n}{2}},(1-a)^{-1}2c\}$. Setting $q:=\max\{a^{\frac{1}{2}},\tilde{q}^{\frac{1}{2}}\}$ and $C_6:=C_5+2C_4$, gives our claim.
\end{proof}

\section{Proof of the main theorem}

Theorem $\ref{theorem4}$ is essential to the following proof. 
\begin{proof}[Proof]
Theorem $\ref{theorem4}$ implies that
\[\|P_x^n-P_y^n\|_{\mathcal{L}}\leq q^n{C_6}(1+\bar{V}(x,y))\quad\text{for }\:x,y\in X\;\text{ and }\:n\in N,\]
where $q$ and $C_6$ are the appropriate constants. Obviously,
\begin{align*}
\begin{aligned}
\|P^n{\mu}-\mu_*\|_{\mathcal{L}}&=\|P^n{\mu}-P^n{\mu_*}\|_{\mathcal{L}}
&=\sup_{f\in\mathcal{L}}\left|\int_X f(z)P^n{\mu}(dz)-\int_Xf(z)P^n{\mu_*}(dz)\right|.
\end{aligned}
\end{align*}
Note that
\begin{align*}
\begin{aligned}
\int_X f(z)P^n{\mu}(dz)-\int_Xf(z)P^n{\mu_*}(dz)&=\int_X\int_Xf(z)P_x^n(dz)\mu(dx)-\int_X\int_Xf(z)P_y^n(dz)\mu_*(dy)\\
&=\int_X\int_X\left(\int_Xf(z)P^n_x(dz)-\int_Xf(z)P^n_{y}(dz)\right)\mu_*(dy)\mu(dx)\\
&\leq\int_X\int_X\|P_x^n-P_y^n\|_{\mathcal{L}}\mu_*(dy)\mu(dx)\\
&\leq q^nC,
\end{aligned}
\end{align*}
where $C:=\int_X\int_XC_6(1+\bar{V}(x,y))\mu_*(dy)\mu(dx)$. Since $C$ is dependant only on $\mu$, the proof is complete.
\end{proof}

\end{document}